\documentclass[10pt]{amsart}

\usepackage{amsmath, amsthm, amssymb, graphicx, enumitem, everypage, datetime, getfiledate} 
\usepackage{wrapfig}

\newcounter{citedtheorems}

\numberwithin{equation}{section}

\newtheorem{defn}{Definition}[section]
\newtheorem{theorem}[defn]{Theorem}
\newtheorem*{theorem-m}{Theorem \ref{main-theorem}}
\newtheorem*{thm-p2a}{Theorem \ref{t:p2a}}

\newtheorem*{thm-seq}{Theorem \ref{t:seq}}

\newtheorem*{thm-m}{Main Theorem}
\newtheorem*{theorem-abs1}{Theorem \ref{ind-theorem}}
\newtheorem*{theorem-abs2}{Theorem \ref{a23}}
\newtheorem*{theorem-abs3}{Theorem \ref{ind-new}}
\newtheorem*{theorem-abs4}{Theorem \ref{m1}}
\newtheorem*{thm-x}{Theorem}
\newtheorem{thm-lit}[citedtheorems]{Theorem}
\newtheorem{defn-lit}[citedtheorems]{Definition}
\newtheorem{fact-lit}[citedtheorems]{Fact}
\newtheorem{fact}[defn]{Fact}

\newtheorem{defn-claim}[defn]{Definition/Claim}

\newtheorem*{defn-in}{Definition \arabic{section}.\arabic{equation}}

\newtheorem*{claim-in}{Claim \arabic{section}.\arabic{equation}}

\newtheorem{conv}[defn]{Convention}

\newtheorem{lemma}[defn]{Lemma}

\newtheorem{expl}[defn]{Example}

\newcommand{\lost}{\L os' }

\newcommand{\br}{\vspace{2mm}}
\newcommand{\scbr}{\vspace{3mm}}

\newcommand{\gee}{\mathcal{G}}

\newcommand{\ml}{\mathcal{L}}
\newcommand{\tlf}{\trianglelefteq}

\newcommand{\rn}{\operatorname{range}}

\newcommand{{\xw}}{\mathbf{w}}

\newcommand{\mcf}{\mathcal{F}}
\newcommand{\mcd}{\mathcal{D}}

\newcommand{\cof}{\operatorname{cof}}

\newcommand{\de}{\mathcal{D}}

\newcommand{\jj}{\mathbf{j}}
\newcommand{\mcp}{\mathcal{P}}

\newcommand{\ba}{\mathfrak{B}}

\newcommand{\lao}{[\lambda]^{<\aleph_0}}

\newcommand{\vp}{\varphi}
\newcommand{\lcf}{\operatorname{lcf}}

\newcommand{\mb}{\mathbf{b}}
\newcommand{\mc}{\mathbf{c}}

\newcommand{\mx}{\mathbf{x}}

\newcommand{\psf}{\operatorname{psf}}
\newcommand{\mfd}{\mathfrak{D}}
\newcommand{\fa}{\mathfrak{D}}
\newcommand{\ld}{\mathbb{D}}

\title{On ultrafilter construction}

\author{Maryanthe Malliaris}\thanks{\emph{Thanks:} Partially supported by NSF-BSF 2051825, and 
based in part on a lecture of the author in Vienna in July 2025.}

\address{Department of Mathematics, University of Chicago, 5734 S. University Avenue, Chicago, IL 60637, USA} 
\email{mem@math.uchicago.edu}

\begin{document}

\begin{abstract}
We give a model-theoretic perspective on regular ultrafilter construction 
in the twentieth and twenty-first century (so far), and  
explain the ``canonical Boolean algebra'' recently 
developed by Malliaris and Shelah. 
\end{abstract}

\maketitle

\br

\hfill{\emph{Dedicated to Saharon Shelah on the occasion of his birthday.}} 

\vspace{5mm}

This article sketches some milestones in regular ultrafilter construction, 
from a model theoretic point of view. 
It points to a book manuscript \cite{MiSh:F2356} in progress, where  
construction of regular ultrafilters is one thread 
in a broader classification program. In later sections of the article, the 
``canonical Boolean algebra'' defined and developed in \cite{MiSh:F2356} 
is motivated and some first theorems are stated.  
Below, it will be useful to know the definition of ultraproduct and the 
fundamental theorem, along with basic model theory such as types and saturation.

Here the discussion is informal, and personal, not aiming to be historically complete, 
rather highlighting certain advances and shifts in understanding.  The hope is to 
provide an accessible, brief introduction to some of these beautiful ideas.

\scbr
\section{Beginnings}

Recall that $\mcd \subseteq \mcp(I)$ is a filter if it is: 
\begin{enumerate}
\item nonempty: (by definition, or by the next item) $I \in \mcd$
\item upward closed: $A \in \mcd$, $A \subseteq B \subseteq I$ implies $B \in \mcd$
\item closed under pairwise, hence finite, intersection: if $A, B \in \mcd$, ~$A \cap B \in \mcd$ 
\item nontrivial: $\emptyset \notin \mcd$.  
\end{enumerate}
When $I$ is a set, often a cardinal, we say $\mcd$ is a filter ``on $I$'' to mean 
``on $\mcp(I)$.''  
Recall that a filter $\mcd$ on $I$ is an \emph{ultrafilter} if for any $A \subseteq I$, exactly one of 
$A$ and $I \setminus A$ belongs to $\mcd$. 
 
This definition makes sense for any Boolean algebra $\ba$, not only $\mcp(I)$, 
replacing 
$I$ by $1_\ba$ in (1), $\emptyset$ by $0_\ba$ in (2), and restricting to elements of $\ba$
in (2) and (3).\footnote{Still we continue the convention of saying 
``filter on $I$'' for sets $I$, often cardinals, and ``filter on $\ba$'' for
Boolean algebras $\ba$, despite the overlap.} 
\footnote{Indeed from the early days of model theory, 
ultrafilters are present not only in ultrapowers 
but also in complete theories and complete types over a set $A$
(ultrafilters on the Lindenbaum algebra of
formulas in $n \geq 0$ free variables, and in the second case parameters from $A$, identified up to
logical equivalence), as is emphasized in \cite{mm-icm}. 
On ultrapowers and their fundamental theorem, see  
\cite{CK}, \S 4.1. 
Ultrafilters have a long pre-history in general topology.}

\br For awhile fix a set $I$. 
We can build ultrafilters on $I$ by Zorn's lemma, 
but this gives us relatively little control over the outcome. 
A more interesting idea was transfinite induction:  

\scbr
\section{Transfinite induction}

\begin{defn} A collection of sets $X \subseteq \mcp(I)$ has the FIP, finite intersection property, 
if the intersection of any finitely many elements of $X$ is nonempty. 
\end{defn}

If $X$ has FIP, it generates a filter 
$\langle X \rangle = \{ Y \subseteq I : $ for some $X_1, \dots, X_k \in X$, $ X_1 \cap \cdots \cap X_k \subseteq Y \}$, by closing under finite intersection and then closing upwards.  

As FIP is preserved under unions of chains, the 
key lemma for a transfinite induction is: 

\begin{lemma} \label{fip-lemma} 
If $\mcd \subseteq \mcp(I)$ has FIP and $A \subseteq I$, then at least one of $\mcd \cup \{ A \}$ or 
$\mcd \cup \{ I \setminus A \}$ has FIP. 
\end{lemma}

\begin{proof}
If there exist $Y_1, \dots, Y_k \in \mcd$ with $Y_1 \cap \cdots \cap Y_k \cap A = \emptyset$, 
and $Z_1, \dots, Z_\ell \in \mcd$ with $Z_1 \cap \cdots \cap Z_\ell \cap (I \setminus A) = \emptyset$, 
then $Y_1 \cap \cdots \cap Y_k \cap Z_1 \cap \cdots \cap Z_\ell = \emptyset$ contradicts 
the FIP for $\mcd$. 
\end{proof}

Thus we can construct an ultrafilter by enumerating the subsets of $I$ and making a decision about them one 
by one while preserving FIP. 
Observe that this construction is inefficient: at many successor steps, applying               
Lemma \ref{fip-lemma}, we did not have a free choice
(for instance, the decision on intersections or complements of sets earlier 
in the enumeration may be forced). 

\scbr
\section{Independent families} 

When do we have a free choice at the successor step? Lemma \ref{fip-lemma} suggests: when both 
$A$ and $I \setminus A$ are consistent with any finite intersections so far. This suggests 
an idea of a basis for the set $\mcp(I)$, or an \emph{independent family of sets}, 
capturing a maximal set of mutually independent such choices. Write $B^1 = B$, 
$B^0 = I \setminus B$.  
This next definition goes back at least to Hausdorff. 

\begin{defn} 
A family of sets $\{ B_\alpha : \alpha < \kappa \} \subseteq \mcp(I)$ is called 
\emph{independent} if for any $n<\omega$, $\alpha_1 < \cdots < \alpha_n < \kappa$ and 
$t_1, \dots, t_n \in \{ 0, 1 \}$, 
\[ \bigcap_{1 \leq \ell \leq n} B^{t_\ell}_{\alpha_\ell} \neq \emptyset. \]
\end{defn} 

Hausdorff proved that over any infinite set $I$ there is an independent family of 
sets of size $2^{|I|}$. Since the property of being an independent family is closed under 
unions of chains, we can consider a maximal such family: restricting our attention 
to elements of this family, we can build an ultrafilter by transfinite induction 
and at every successor step, Lemma \ref{fip-lemma} presents us with a real choice.   

Now, however, we can identify each $B_\alpha$ with its characteristic function. 
Call the family of functions 
$\gee \subseteq {^I 2}$ \emph{independent} if for every $n <\omega$ and every 
distinct $g_0, \dots, g_{n-1} \in \gee$ and every $j_0, \dots, j_{n-1} < 2$, 
\[ \{ t \in I : g_0(t) = j_0 \land \cdots \land g_{n-1}(t) = j_{n-1} \} \neq \emptyset. \]
This shift in perspective suggests the next idea: why bound the range at 2?

\begin{theorem}[Engelking-Karlowicz...] \label{t:ind-family} 
For any $\lambda \geq \kappa \geq \aleph_0$ there exists $\mcf \subseteq {^\lambda \kappa}$ which is 
an independent family of functions of size $2^\lambda$. 
This means: each $f \in \mcf$ has 
domain $\lambda$ and range $\kappa$, and for any $n <\omega$ and every 
distinct $g_0, \dots, g_{n-1} \in \gee$ and every $j_0, \dots, j_{n-1} < \kappa$, 
\[ \{ t \in I : g_0(t) = j_0 \land \cdots \land g_{n-1}(t) = j_{n-1} \} \neq \emptyset. \] 
\end{theorem}
Looking towards later inductions, we can also define 
``independent modulo the filter $\de$'' by changing ``$\neq \emptyset$'' in the last line of
$\ref{t:ind-family}$ to be ``$\neq \emptyset \mod \de$.''  
(A set $A$ is equal to $\emptyset$ modulo
a filter if $(I \setminus A) \in \de$.)
\scbr 
\section{Regularity}

\begin{defn}
A filter $\de$ on $I$, $|I| = \lambda$ is \emph{regular} if there exists 
$\{ X_\alpha : \alpha < \lambda \} \subseteq \de$, called a \emph{regularizing family}, 
such that for any infinite $\sigma \subseteq \lambda$, 
\[ \bigcap_{\alpha \in \sigma} X_\alpha ~ = \emptyset. \]
\end{defn}

Consistently, all ultrafilters are regular. 
They are easy to build: fix $\lambda$, let 
$I = [\lambda]^{<\aleph_0}$ be the set of finite subsets of $\lambda$, and 
observe that $\{ \{ u \in I : \alpha \in u \} : \alpha \in \lambda \}$ 
is a regularizing family and has the FIP. 

A useful property of regular ultrafilters is that regular ultrapowers 
(=ultrapowers where the ultrafilter in question is regular) have the full size of the Cartesian
power: if $M$ is an infinite model and $\de$ is a regular ultrafilter on $I$,
then
\[ |M^I/\de| = |M|^{|I|}. \]

An early use of regular ultrafilters was to prove the compactness theorem 
via ultraproducts. Let $\de$ be a regular ultrafilter on $I$ 
extending the regularizing family just built. Let $T = \langle \vp_\alpha : \alpha < \lambda \rangle$ be a theory 
such that every finite subset of $T$ has a model. 
For each $u \in I$, that is, for each $u \in [\lambda]^{<\aleph_0}$, choose 
$M_u \models \{ \vp_\alpha : \alpha \in u \}$ to be a model of the corresponding 
finite subset of $T$.  
Then each $\vp_\alpha$ will hold in $M_u$ for at least those 
$\{ u \in I : \alpha \in u \} \in \de$, 
so by the fundamental theorem of ultraproducts, the ultraproduct is a model of $T$. 

It turns out that in regular ultrapowers, any \emph{countable} 
partial type --  
that is, a set of formulas $\{ \vp_\alpha(x,\bar{a}_\alpha) : \alpha < \omega \}$  
with the property that any finitely many are simultaneously satisfiable --   
will be realized. So regular ultrapowers of countable theories are $\aleph_1$-saturated. Achieving more saturation is a much more subtle question.  
We will return to this.

\scbr 
\section{Characterizing elementary equivalence}

An important problem in the 1960s was the following. Suppose $M, N$ are models in the same 
vocabulary $\ml$. Defining ``$M \equiv N$'' (the same first-order sentences hold in both models) 
requires many steps: you need a notion of logic, of 
sentences, of truth in a model. Is there an ``algebraic'' 
or ``outside'' characterization of the equivalence relation on models 
given by $\equiv$? 

This was answered by Keisler under GCH, and some time later Shelah removed GCH, so 
the answer is known as the Keisler-Shelah isomorphism theorem: \emph{$M \equiv N$ 
if and only if $M$, $N$ have isomorphic ultrapowers.} Keisler's proof under GCH 
goes as follows. Fix $\lambda \geq \max \{ |M|, |N|, |\ml| \}$. Then:
\begin{enumerate}
\item[(a)] There exists a regular \emph{good} ultrafilter $\de$ on $\lambda$ (Keisler's proof used GCH). 
\item[(b)] If $\de$ is a regular good ultrafilter on $\lambda \geq |\ml|$ and $M$ is an $\ml$-structure, 
then $M^\lambda/\de$ is $\lambda^+$-saturated. 
\item[(c)] By the fundamental theorem of ultraproducts $M^\lambda/\de \equiv M$ 
and $N^\lambda/\de \equiv N$, by hypothesis $M \equiv N$, and $\equiv$ is transitive. 
\item[(d)] By regularity, $| M^\lambda/\de | = | N^\lambda/\de | = 2^\lambda$.    
\item[(e)] Assuming $2^\lambda = \lambda^+$, (b)+(d) imply   
$M^\lambda/\de$ and $N^\lambda/\de$ are both saturated. 
\item[(f)] In conclusion, $M^\lambda/\de$, $N^\lambda/\de$ are  
are elementarily equivalent, saturated models of the same size, hence they are isomorphic.  
\end{enumerate}

We will return to Shelah's proof in due course.  
Let us consider this very interesting property called ``good''. 

\scbr 
\section{Good}  

\begin{defn}
A filter $\de$ on $|I|$, $|I| = \lambda$ is called \emph{good} if any monotonic
$f: [\lambda]^{<\aleph_0} \rightarrow \de$ has a multiplicative refinement, where:

\begin{enumerate} 
\item $g$ refines $f$ means $g: [\lambda]^{<\aleph_0} \rightarrow \de$ and 
$g(u) \subseteq f(u)$ 
\item monotonic, really antimonotonic, means $u \subseteq v \implies g(u) \supseteq g(v)$  
\item multiplicative means $g(u) \cap g(v) = g ( u \cup v )$  
\end{enumerate} 
for all $u,v \in [\lambda]^{<\aleph_0}$. 
\end{defn}

It is not hard to see that any nonprincipal ultrafilter on a countable set is good.  

Why do such ultrafilters exist in general, and why do they produce saturation in ultrapowers? 
We start with the first question by sketching 
Kunen's ZFC proof (after Keisler's GCH proof) that good regular ultrafilters exist on any  
infinite $\lambda$. See \cite{kunen}, \cite{keisler-2}.

Looking towards an induction, recall the definition above of an 
\emph{independent family of functions} from $\lambda$ to $\lambda$. 
First prove an analogue of Lemma \ref{fip-lemma}: if $\mcf$ is independent modulo the filter
$\de$ on $I$, and $A \subseteq I$, we can decide $A$ at the cost of throwing away
 finitely many elements of $\mcf$, with the remainder staying 
independent modulo the increased filter. 

Now we have a framework for transfinite induction where at any given stage, the
the elements of the remaining independent family can be seen as representing
independent partitions of $I$
(the fibers of each $f$) into sets which are not yet large or small.

To ensure the outcome is a regular ultrafilter, the base case is given by:

\begin{fact} \label{f:234}
For any $\lambda \geq \kappa \geq \aleph_0$, there exist a regular filter
$\de_0$ on $\lambda$ and a family $\mcf \subseteq
{^\lambda \kappa}$ of functions independent modulo $\de_0$.
\end{fact}

So we begin with a regular filter $\de_0$ on $\lambda$ and $\mcf$ 
as in \ref{f:234} for $\kappa = \lambda$. Let $\langle A_\beta : \beta < 2^\lambda \rangle$ enumerate 
$\mcp(\lambda)$, and let $\langle g_\beta : \beta < 2^\lambda \rangle$ enumerate 
all functions from $[\lambda]^{<\aleph_0}$ into $\mcp(I)$, each occurring cofinally often.  
By induction on $\alpha < 2^\lambda$, build an increasing continuous\footnote{At limits 
take unions.} sequence $\langle \de_\alpha : \alpha < 2^\lambda \rangle$ on filters, 
and a decreasing continuous\footnote{At limits take intersections} sequence 
$\langle \mcf_\alpha : \alpha < 2^\lambda \rangle$ of families of functions from $\lambda$ to $\lambda$, so that:
\begin{enumerate}
\item[(a)] $\mcf_\alpha$ is independent modulo $\de_\alpha$
\item[(b)] $| \mcf_{\alpha} \setminus \mcf_{\alpha+1}| < \aleph_0$
\item[(c)] if $\alpha = 2\beta$, $\de_{\alpha+1}$ decides $A_\beta$
\item[(d)] if $\alpha = 2\beta+1$, if $\rn(g_\beta) \subseteq \de_\alpha$, 
then there is $f: [\lambda]^{<\aleph_0} \rightarrow \de_{\alpha+1}$ 
which is a multiplicative refinement for $g_\beta$.    
\end{enumerate} 
Item (c) is a simple way to ensure an ultrafilter.  
Item (d) makes the essential use of the freedom guaranteed by the independent family: 
we can add the range of a well-chosen multiplicative refinement 
to the filter while maintaining FIP (at the cost of a single independent function). 
Since the induction has cofinality at least $\lambda^+$, enumerating 
all $g$'s cofinally often will ensure that if the range of $g$ is in the final 
ultrafilter, it will have been handled at some inductive stage.  
Thus we obtain a good regular ultrafilter on $\lambda$, for any 
infinite $\lambda$, in ZFC.

\scbr
\section{Keisler's order}

In \cite{keisler} Keisler defined\footnote{Perhaps by extracting the
ZFC information from his version of the isomorphism theorem.}
\footnote{Not to be confused with the Rudin-Keisler order on ultrafilters.}
an order on complete countable theories which became central for later work 
(and will be a focus for us here).

Say ``$\de$ saturates $M$'' if $\de$ is a regular ultrafilter on $\lambda$ and $M^\lambda/\de$
is $\lambda^+$-saturated. Keisler proved that if
$\de$ is a regular ultrafilter, $T$ is a complete countable theory, and $M, N$ are infinite
models of $T$, then $\de$ saturates $M$ iff $\de$ saturates $N$.  In other words, saturation depends
on the regular ultrafilter and on the theory, but not the model chosen.
Let us say ``$\de$ saturates $T$'' if $T$ is a complete countable theory
and $\de$ saturates $M$ for some, equivalently every, $M \models T$.
This allows Keisler to define:

\begin{defn}[Keisler's order]
\[ T_1 \tlf T_2 \]
if for every regular ultrafilter $\de$, if $\de$ saturates $T_2$ then $\de$ saturates
$T_1$.
\end{defn}

We will now look more closely at this very interesting definition.

\scbr 
\section{Distributions and saturation}

Let $I$ be an infinite set (as a convention, $|I| = \lambda$),  
$T$ a countable theory, $M$ an infinite model of $T$, 
and $\de$ a regular ultrafilter on $I$. 

\begin{conv} For the rest of the article all theories are complete and countable 
and have only infinite models.\footnote{This is not always needed; often 
``$|\ml| \leq \lambda$ suffices. Among other things this saves having 
to distinguish between ``$\lambda^+$-saturated'' [every type over $\lambda$ 
parameters is realized] and ``$\lambda^+$-compact'' [every partial type 
consisting of $\lambda$ formulas, and hence also at most $\lambda$ parameters,  
is realized].}
\end{conv} 

Recall that a model is $\lambda^+$-saturated if every type over every 
set of $\leq \lambda$ parameters is realized. For types over 
small sets, ultrapowers have strong saturation properties. For example, any 
nonprincipal ultrafilter on $\omega$ produces $\aleph_1$-saturated models, and 
any regular ultrafilter on any infinite set produces $\aleph_1$-saturated models. 

What happens above $\aleph_1$?  Let $\de$ be a regular ultrafilter on $I$, $|I| = \lambda$. 
To see the mechanism, first consider a countable example. 

\begin{expl} 
Let $M = (\mathbb{N}; <)$,  $N = M^I/\de$, 
let $n \mapsto [n]$ be the diagonal embedding of $M$ into $N$, and 
consider the type $p(x) = \{ x > [n] : [n] \in \mathbb{N} \}$ in the ultrapower 
expressing existence of a nonstandard element. 
\end{expl}

This type is necessarily realized because $N$, being a regular ultrapower, has its full size 
$|M|^{|I|} > \aleph_0$; but we don't realize $p$ by satisfying all conditions everywhere. Rather, consider countably many elements of a regularizing family, $X = \{ X_n : n < \omega \} \subseteq \de$ and the map $d: p \rightarrow X$ given by: 
\[  \mbox{``}x > [n] \mbox{''} ~ \mapsto X_n \]
so for each $t \in I$, we have a finite set $p_t = \{ x > n  : t \in X_n \}$ of 
formulas in $M$ to satisfy. Let $c_t$ satisfy $p_t$ in $M$, and then 
$c := \langle c_t : t \in I \rangle /\de$ satisfies $p$ in $N$ by the fundamental theorem of 
ultraproducts. 
In this example, we realize $p$ by scattering the conditions 
across $I$ in such a way 
that: 
\begin{itemize}
\item each condition is addressed on a $\de$-large set.
\item at each $t \in I$, there are at most finitely many conditions to consider. 
\item at each $t \in I$, the conditions assigned to $t$ can be simultaneously satisfied.
\end{itemize} 
As this can be achieved we see how the type can be realized. 

In generality, types are not always realized precisely because 
we cannot always create such an assignment. Let us describe the problem 
of realizing a type over a set of size $\leq |I|$ 
as follows.\footnote{The rest of this section moves into modern 
language for presenting the order, following conventions in \cite{mm-thesis}.}  
Given $M$, $N = M^I/\de$, $|I| = \lambda$, and a 
type $p = \{ \vp_\alpha(x;\bar{a}_\alpha) : \alpha < \lambda \}$ of $N$, 
the fundamental theorem of ultraproducts (``\lost theorem'')  
says for each finite $u \subseteq \lambda$, 
there is a set $\L(u) \in \de$ such that   
$t \in \L(u) \implies M \models \exists x \bigwedge \{ \vp_\alpha(x,\bar{a}_\alpha[t] : \alpha \in u \}$. 
Fix some regularizing family $\{ X_\alpha : \alpha < \lambda \} \subseteq \de$. 
Consider the map $d: [\lambda]^{<\aleph_0} \rightarrow \de$ given by:
\[ u \mapsto \L(u) \cap \bigcap \{ X_\alpha : \alpha \in u \}. \]
This is an example of a \emph{distribution of $p$}. The features of a distribution are:
\begin{itemize}
\item $d: [\lambda]^{<\aleph_0} \rightarrow \de$~ (here we could have said $d: [p]^{<\aleph_0} \rightarrow \de$). 
\item $d$ is monotonic: $u \subseteq v$ implies $d(v) \subseteq d(u)$. 
\item $d(u) \subseteq \L(u)$: each finite set $u$ of conditions is addressed on a $\de$-large set, on which 
they all have a common solution.  
\item for each $t \in I$, $P_t := \{ \alpha < \lambda : t \in d(\{\alpha\}) \}$ is finite.  
\item (accuracy\footnote{This helps in characterization, but will mainly work behind the scenes in this paper.}) for each $t \in I$ and $v \subseteq P_t$, $t \in d(v)$ iff 
$t \in \L(v)$.  
\end{itemize}
The key point is that the various conditions in $P_t$ need not have a common solution: 
by monotonicity $d(u) \cap d(v) \subseteq d(u \cup v)$, 
but equality is not guaranteed to hold unless $d$ is multiplicative.\footnote{For instance, 
if $\vp_\alpha(x; a_\alpha,b_\alpha) = a_\alpha < x < b_\alpha$, 
$\vp_\beta(x; a_\beta, b_\beta) = a_\beta < x < b_\beta$ are two formulas in an ultrapower of dense 
linear order, $t \in d(\{\alpha\})$ tells us ``$(a_\alpha[t], b_\alpha[t])$ is a nonempty interval in $M$,''
$t \in d(\{\beta\})$ tells us ``$(a_\beta[t], b_\beta[t])$ is a nonempty interval in $M$,'' 
but $t \in d(\{\alpha, \beta\})$ asks for the a priori stronger condition that 
they are overlapping intervals in $M$, which will happen for most, but not necessarily all, 
$t \in d(\{\alpha\}) \cap d(\{\beta\})$.  It is 
a nice exercise to show that $p$ is realized if and only if some (equivalently, every) 
distribution $d$ of $p$ has a multiplicative refinement.}  

So we arrive to goodness. Rephrasing a theorem of Keisler in our language:  

\begin{theorem}
If $M$ is a model in a countable language and $\de$ is a regular good ultrafilter on $I$, 
$|I| = \lambda$, then every distribution of every type $p$ over a set of size $\leq \lambda$ has 
a multiplicative refinement, \emph{hence} $M^I/\de$ is $\lambda^+$-saturated. 
\end{theorem}

\scbr 
\section{Minimal and maximal}

Keisler noticed that Keisler's order had a minimum class:
\[ \mathbf{T}_{\operatorname{min}} = \{ T : \mbox{ every regular ultrafilter saturates $T$} \} \] 
which includes any theory, such as algebraically closed fields of
fixed characteristic, whose uncountable models are all saturated (hence 
this class is nonempty). 

Keisler also noticed that the order had a maximum class: 
\[ \mathbf{T}_{\operatorname{max}} = \{ T : \mbox{ a regular ultrafilter $\de$ saturates $T$ if and only if $\de$ is good} \}. \]
To see this class is nonempty, first notice that any theory requires a certain amount of goodness from 
an ultrafilter: any monotonic function which arises as a distribution for one of its types must have a 
multiplicative refinement. The point is that for some sufficiently expressive theories, like 
Peano arithmetic or the theory of atomless Boolean algebras, all relevant functions
arise as distributions, and so asking that a regular ultrafilter saturate such a theory $T$ 
requires it to be good (hence able to saturate any other complete countable theory).

\scbr
\section{Stable theories}

When building regular ultrafilters via independent families,  
how to realize some types while omitting others? (In the language of goodness, how to 
produce regular ultrafilters which are not good? In the language of Keisler's order, how to 
distinguish theories?) A first beautiful answer is given by 
Shelah in \emph{Classification Theory}, chapter VI (after his ZFC proof of 
the isomorphism theorem, outlined for interested 
readers in \cite{Sh:a} exercises 3.1-3.5).  

Recall that when $D$ is a filter on $I$, not 
necessarily an ultrafilter, we can still define reduced products which are not necessarily ultraproducts, identifying elements if they agree on a set which we have already decided is large.   
In the inductive constructions above, elements of independent families keep track of partitions of $I$. 
Shelah suggests considering them instead as elements of the Cartesian power $M^I$, and looking 
at how the reduced power $M^I/D$ evolves as we add sets to $D$.  

In this sketch, our model will be fixed as $M = (\omega, <)$. 
Define\footnote{In \cite{Sh:c}, ``psf'' is $\mu(\de)$; $\lcf$ abbreviates lower cofinality, aka reverse
cofinality.}
 two invariants of a regular ultrafilter $\de$ on $\lambda$:  

\begin{itemize}
\item the minimum size of a pseudofinite set modulo $\de$: 
\[ \psf(\de) \] 
which for our purposes, can be computed in $(\omega, <)^\lambda/\de$ as the minimum cardinality  
$\{ a : a < b \}$ as $b$ ranges over nonstandard numbers.\footnote{Identifying $\omega$ 
with its diagonal embedding in the ultrapower, a number $a \in (\omega,<)^\lambda/\de$ is 
called nonstandard if $a>n$ for every $n \in \omega$.}
\item the co-initiality of the set of nonstandard numbers 
in $N = (\omega, <)^\lambda/\de$, called   
\[ \lcf(\omega, \de) \] 
which for us means there is an unfilled $(\omega, \lcf(\omega,\de))$-cut in 
$N$, one side of which is given by the diagonal embedding of $\omega$.  
\end{itemize}
Note that $\lcf(\omega, \de) \leq \psf(\de)$ because $\lcf$ describes a sequence in
the set whose size is described by $\psf$.

Suppose $D$ is a regularizing family on $I$ and 
$\mcf \subseteq{^\lambda \omega}$ is an independent family modulo $D$ of size $2^\lambda$.  
To emphasize, $\mcf$ has range $\omega \leq \lambda$,  
hence each $f \in \mcf$ 
also belongs to the cartesian power $M^I$. The construction 
will involve a choice of initial family $\mcf_0 \subseteq \mcf$ 
along with an ordinal $\rho$ of cardinality $|\mcf_0|$ and uncountable 
cofinality, both fixed later. 

By induction on $\alpha < \rho$, build an increasing continuous 
sequence $\langle D_\alpha : \alpha < \rho \rangle$ of filters extending $D$,
and a decreasing continuous 
sequence $\langle \mcf_\alpha : \alpha < \rho \rangle$ of families of 
functions from $\lambda$ to $\omega$ starting with $\mcf_0$, so that:
\begin{enumerate}
\item[(a)] $\mcf_\alpha$ is independent modulo $D_\alpha$
\underline{and} $D_\alpha$ is maximal for this property\footnote{If $\langle D_\ell : \ell < \gamma \rangle$ is an increasing chain of ultrafilters, and $\mcf$ is independent modulo each $D_\ell$, then 
also $\mcf$ is independent modulo their union; so we can without loss of generality ensure 
the filter is maximal for the property that $\mcf$ remains independent.}  
\item[(b)] $|\mcf_\alpha \setminus \mcf_{\alpha+1}| \leq 1$
\item[(c)] at each stage $\alpha$, we                     
choose some $g_\alpha \in \mcf_\alpha$ and add to $D_{\alpha+1}$ the 
information that $g_\alpha$ is nonstandard: 
\[ D_{\alpha+1} \supseteq \langle D_\alpha \cup \{ \{ t \in I : g_\alpha(t) \geq n \} : n \in \omega \}. \]  
\end{enumerate}
In the background is substantially upgraded technology for 
dealing with independent families of functions and reasoning about the freedom remaining. 
As a result, one can show 
\begin{quotation}
\noindent (1) if $\rho$, the cofinality of the construction, is uncountable 
(more precisely: is greater than the range of the independent family), 
and $\mcf_\rho = \emptyset$, then (a) is enough to ensure we have ended up with an ultrafilter. 
\end{quotation} 
Moreover, if $g_\alpha$ is set to be nonstandard at stage $\alpha+1$, then: 
\begin{quotation}
\noindent (2) if $f \in \mcf_{\alpha+1}$, i.e. $f$ still remains independent, then for any 
future ultrafilter $\de \supseteq D_{\alpha+1}$, we have 
\[ \{ t \in I : g_\alpha(t) > f(t) \} \in \de. \]

\noindent (3) if $|\mcf_{\alpha+1}| = \kappa$, then for any future 
ultrafilter $\de \supseteq D_{\alpha+1}$, 
in the ultrapower $N = M^I/\de$ we have that 
\[ | \{ b \in N : b < [f/\de] \} | \leq \kappa. \]
\end{quotation} 
Now $\lcf$ depends on the cofinality of $\rho$: we walk 
along the transfinite sequence setting elements to be nonstandard as we go, and 
pushing a wave of elements (which remain potentially standard) in front of us. 
So $\langle g_\alpha/\de : \alpha < \rho \rangle$ will form a strictly decreasing sequence 
of nonstandard numbers, of cofinality $\cof(\rho)$. 
If $h \in {^\lambda \omega}$ is such that $h/\de$ is nonstandard, 
necessarily it is above some $g_\gamma/\de$, 
since we just have to wait for the countably many sets attesting to $h$ being nonstandard 
to be added to $\de$, and $\cof(\rho) \geq \aleph_1$. 
This shows $\langle [g_\alpha/\de] : \alpha < \rho \rangle$ 
is cofinal in the nonstandard numbers.   

Suppose we want to make $\psf$, and hence also $\lcf$, small: 
start with $\mcf_0$ say of size continuum.  
Suppose we want to make $\psf$ large and $\lcf$ small: start with $\mcf_0 = \mcf$ 
of size $2^\lambda$ but enumerate it in $\kappa$ blocks each of size $2^\lambda$, 
so that $\cof(\rho) = \kappa$.  
Summarizing:  

\begin{theorem}[Shelah, \cite{Sh:a} VI.3.10] \label{t:sizes} 
Suppose $\lambda$ is an infinite cardinal, 
$\mu = \mu^{\aleph_0}$, $\kappa$ is regular and uncountable, and 
$\kappa \leq \mu \leq 2^\lambda$. There is a regular ultrafilter $\de$ on $\lambda$ 
with $\psf(\de) = \mu$ and $\lcf(\omega,\de) = \kappa$.  
\end{theorem}

This is a theorem on ultrafilters, but part of its power 
comes from the way it works with the theory of stability developed in \cite{Sh:a} chapters I-III, 
the major model theoretic dividing line of the last half-century.  
Shelah proves as a consequence of his development of stability 
that a model of a stable theory is $\lambda^+$-saturated if it is $\kappa(T)$-saturated
 and every maximal indiscernible set is large.  
Very briefly, this means that in regular ultrapowers of stable theories, 
saturation depends on whether or not pseudofinite sets are large, that is, 
whether $\psf$ is large, and for saturation of  
for unstable theories, it is necessary (but not sufficient) that $\lcf$ is large. 
So the union of the first two classes in Keisler's order is precisely the stable theories,  
with a division at the finite cover property, which determines whether the theory is 
sensitive to $\psf$ being small. There is also a strict division between 
stable and unstable in Keisler's order, witnessed by ultrafilters with large 
$\psf$ and small $\lcf$. Hence Keisler's order independently detects stability.  

\br

I read these theorems carefully as a graduate student and found them very inspiring, 
both because of what they show can be done and the remaining open problems they raised.

\scbr
\section{Locality}

Here I mention some work from my thesis and early papers which will be invoked below. 

First, Keisler's order is local: if $\de$ a is a regular ultrafilter on $I$, 
$|I| = \lambda$, $T$ is a complete countable theory, $M \models T$ and 
$M^I/\de$ is $\lambda^+$-saturated for $\vp$-types for all formulas $\vp$, then 
it is $\lambda^+$-saturated \cite{mm1}. 
So we are justified in looking formula by formula, going one $\vp$ at a time below. 
(Note that this doesn't say that to each theory we associate a single formula which is 
a bellwether for its saturation, but simply that any instance of non-saturation must 
be witnessed by omitting a $\vp$-type for some $\vp$.) 

This re-focuses attention on the patterns of consistency and inconsistency 
visible in the parameter space of formulas. 

Second, here is a measurement of regularizing sets which is visible 
to first-order theories. If $\de$ is a regular filter on $I$, $|I| = \lambda$, and 
$\mathbf{X} = \{ X_\alpha : \alpha < \lambda \}$ is a regularizing family, by definition for every 
$t \in I$ there is a finite $n_t \geq | \{ \alpha < \lambda : t \in X_\alpha \} |$. 
Let $n_* = \prod n_t/\de$ be the corresponding nonstandard natural number. Then let us say 
``$\mathbf{X}$ is below $n_*$.'' Say that $\de$ is \emph{flexible} if  
for any nonstandard $n_*$, $\de$ contains a regularizing family below $n_*$. This 
property was introduced and studied in \cite{mm4} and turned out, fortuitously, 
to coincide with a property called OK studied by Kunen and Dow.\footnote{This interesting 
property says that any monotonic $f: [\lambda]^{<\aleph_0} \rightarrow \de$ which, in addition, 
satisfies that $|u| = |v|$ implies $f(u) = f(v)$, has a multiplicative refinement. This 
seems a priori much simpler than good but this was not obvious.}   
Moreover, it is detected by model theory: if, relative to the theory $T$, the formula 
$\vp$ is non-simple or non-low, then in order for $\de$ to be good for $T$ it is necessary 
that $\de$ be flexible \cite{mm5}.  
Note that, even without the connection to OK, this already 
tells us that good implies flexible (since whether or not such regularizing families 
exist is visible to a type, hence to a question of multipicative refinement).  

Below, this will mean that 
we can prove failure of goodness by failure of flexibility. After the connection to OK, 
it also means we can potentially use model-theoretic tools to approach questions such as  
separating OK and good, asked by Dow in 1985.  

This is a good place to mention something I found particularly beautiful about the 
problem of Keisler's order as it was developed and advanced in Shelah's book 
(this is both a response to what was written and an interpretation).  
Theories are a priori very complicated objects, and ultrafilters may allow us 
to separate the major differences from the minor ones. 
Ultrafilters are perhaps difficult to see directly, but we can look 
obliquely by observing their effect on theories.  

\scbr 
\section{Separation of variables}

This is the subject of \cite{MiSh:999} proved during a wonderful year in Jerusalem. 

Separation of variables says 
that we can reframe the problem of building regular 
ultrafilters on $\lambda$ as a problem about building ultrafilters (regular or not) 
on Boolean algebras. To explain this statement requires several steps. 

When building a regular ultrafilter on $I$, $|I|=\lambda$ by induction on (say)
$\alpha < 2^\lambda$, we can ask at some intermediate stage $\alpha$ what is the
quotient Boolean algebra $\mcp(I)/D_\alpha$. For instance, if $\mcf_\alpha \subseteq {^\lambda \kappa}$,
$|\mcf_\alpha|=2^\lambda$ and $D_\alpha$ is maximal for $\mcf_\alpha$ being independent, then
the quotient is essentially the completion of the free Boolean algebra generated by
$2^\lambda$ independent antichains each of size $\kappa$.
Step 1 says that we can arrange for the quotient to be any complete $\ba$ which looks minimally like $\mcp(I)$ as measured by size and width: 
\begin{quotation} 
\noindent \emph{Step 1.} If $\ba$ is a complete Boolean algebra of size $\leq 2^\lambda$ 
and the $\lambda^+$-cc, 
there is a regular excellent [see below] filter $\de_0$ and a surjective homomorphism  
$\jj: \mcp(I) \rightarrow \ba$ with $\jj^{-1} = \de_0$.  
\end{quotation}
Next observe: 
\begin{quotation}
\noindent \emph{Step 2.} Suppose $\de_0$, $\jj$, $\ba$ are from Step 1. 
If $\mfd$ is any ultrafilter on $\ba$, this induces an ultrafilter $\ld \supseteq \de_0$ on 
$I$ by: $A \in \ld$ if $\jj(A) \in \mfd$. 
The choice of $\mfd$ and a model $M$ thus determines an 
ultrapower $N = M^I/\ld$ (call such an $N$ an \emph{enveloping ultrapower} for 
$\ba$). We will call such a constellation of objects an 
\emph{ultrafilter frame} $\mathfrak{u} = (I, \de_0, \mfd, \ld, \ba, j)$.   
\end{quotation}
What are the appropriate traces of types in $\ba$? Here is a literal definition:  
\begin{quotation}
\noindent \emph{Step 3.} 
Say that $\bar{\mb} = \langle \mb_u : u \in [\lambda]^{<\aleph_0} \rangle$ is a 
\emph{possibility pattern} for $(T, \vp)$ if in some enveloping ultrapower there 
is a $\vp$-type $p$ and a distribution $d: [\lambda]^{<\aleph_0} \rightarrow \de$ of $p$ 
and $\mb_u = \jj(d(u))$ for all $u \in [\lambda]^{<\aleph_0}$.    
\end{quotation}
In other words, we look for the images of distributions of types in some enveloping  
ultrapower. 
There is an equivalent combinatorial 
characterization.\footnote{The combinatorial condition on $\bar{\mb}$ asks:
suppose $U \subseteq \lambda$ is finite and 
suppose $\mc \in \ba$ is a nonzero element contained in some nonempty region of the
Venn diagram made by the elements $\mb_{\{\alpha\}}$, for
$\alpha \in U$: so either $\mc \leq \mx$ or $\mc \cap \mx = 0$ for $\mx \in \{\mb_v : v \subseteq U \}$.
Can we always find parameters $\bar{a}_\alpha$ for instances of $\vp(x,\bar{y})$ in some $M \models T$
whose consistency and inconsistency exactly reflects what is seen by $\mc$, i.e.,     
for $v \subseteq U$, ~
$\exists x \bigwedge \{ \vp(x,\bar{a}_\alpha) : \alpha \in v \}$ if and only if $\mc \leq \mb_v$?
}
Now we check: has $\ba$ lost essential information? Suppose we built our $\mfd$ on $\ba$ 
to be good, basically as before except that we work on $\ba$ rather than on $I$:  
list all possibility patterns cofinally often, and inductively add 
multiplicative refinements to any pattern which has been added to the filter in progress. Does 
this create goodness for $\de$? A priori no: if $\jj$ takes $B_u, B_v, B_{u \cup v}$ to 
$\mb_u, \mb_v, \mb_{u \cup v}$ respectively, then $\mb_u, \mb_v, \mb_{u \cup v} \in \mfd$ implies 
$B_u, B_v, B_{u \cup v}$ in $\de$, but  ``$\mb_u \cap \mb_v = \mb_{u \cup v}$'' in $\ba$ only 
guarantees that ``$B_u \cap B_v = B_{u \cup v} \mod \de_0$'' (which trivially follows from 
these three sets belonging to $\de$).      
This is where ``$\de_0$ is excellent'' comes in: it is a condition which allows us to 
``simultaneously and uniformly'' erase errors from certain sequences of subsets of $I$, 
so that ``$\emptyset \mod \de_0$'' becomes ``$\emptyset$''. In this case, we
 show that realization of types does transfer. Part of the 
work\footnote{It is shown in our paper that `good' can replace `excellent' in the 
statement of separation of variables given here, so we leave the description informal here.} 
is to 
give an appropriate definition of the errors to be erased (multiplicative 
modulo $\de_0$ becomes multiplicative, \emph{but} we do not add intersections for 
all sequences of elements of $\de_0$).  

Summing up, separation of variables says that we can transfer 
construction problems in Keisler's order to problems of building 
ultrafilters which can solve ``possibility patterns'' 
(the representations of types) 
on any reasonable, complete Boolean algebra of our choosing: 
 
\begin{theorem}[Separation of variables, informal statement of \cite{MiSh:999}] 
Fix $I$, $|I| = \lambda$. 
\begin{enumerate}
\item Given any complete $\ba$ which has size $\leq 2^\lambda$ and 
the $\lambda^+$-c.c., there exists a regular excellent filter $\de_0$ on $I$ and 
a surjective homomorphism $\jj: \mcp(I) \rightarrow \ba$ with $\jj^{-1}(1) = \de_0$. 
\item Then, given any ultrafilter frame $\mathfrak{u} = (I,  \de_0, \fa, \ld, \ba, \jj )$,  
for any complete countable $T$ the following are equivalent:
\begin{enumerate}
\item any sequence $\langle \mb_u : u \in \lao \rangle$ of 
elements of $\fa$ which is a possibility pattern for $T, \vp$ has, in $\fa$, a multiplicative 
refinement;   
\item for any $M \models T$, $M^I/\ld$ is $\lambda^+$-saturated.
\end{enumerate} 
\end{enumerate} 
\end{theorem}

\br
This theorem reframed our approach to ultrafilter construction. Here are two early consequences. 
First, by choosing $\ba$ to be the completion of a free Boolean algebra generated by  
$2^\lambda$-many independent countable antichains,  
we see a division in Keisler's order  
within the simple unstable theories, 
with stable and the random graph on one side, 
and non-low or non-simple on the other.\footnote{This 
statement incorporates later understanding.}

Second, we can now use set theoretic hypotheses which are a priori incompatible with regularity 
to build $\mfd$ on $\ba$, and still end up with $\de$ a regular ultrafilter! Thus 
assuming existence of a measurable cardinal, we prove there are regular ultrafilters on 
all sufficiently large $\lambda$ which are flexible and not 
even 
$(2^{\aleph_0})^+$-good, giving a conditional answer to an old question of Dow mentioned above.\footnote{In ZFC, at the time of writing, this remains open.}

\scbr
\section{Continuum many classes}

There were a number of advances on Keisler's order building on separation of variables. 
For the purposes of this paper and the narrative of regular ultrafilter construction, 
let me focus on a theorem of \cite{MiSh:1167}. There 
we proved that Keisler's order has the maximum number of classes, continuum many, 
definitively answering Keisler who in 1967 had asked if there were more than two.

The many classes arise in an interesting way that suggest the order still has 
much to teach us about model theory. They are within the
simple unstable theories with trivial forking; on other large parts of the 
model theoretic map, Keisler's order has few classes.  We have already seen that stable theories fall into only 
two classes; more generally, all NIP theories and indeed all theories with $SOP_2$ are
contained in three classes, see \cite{MiSh:998}. 
 
Perhaps Keisler's order is sensitive to fine calibrations of randomness within 
simplicity just as long as there is not enough other structure to add noise. And perhaps, also, 
simple theories are more interesting than has been understood.  

To say a few words about the construction:   
Separation of variables in some sense says that if we produce a complete Boolean 
algebra $\ba$ (of the right size and width) along with an ultrafilter $\mfd$ on $\ba$ 
\emph{then} it is possible to find an $I$, $\de_0$, $\jj$ which combine with 
$\ba$ and $\mfd$ to give us a regular $\de$. So the construction problem focuses 
on Boolean algebras and the ultrafilters on them.   
In order to find these continuum many classes, we first defined a suitable family of 
continuum many theories, based on combinations of unary predicates and random graphs.   
Then we built ultrafilters and Boolean algebras 
together by induction for each given theory, essentially inductively extending 
the Boolean algebra at a given successor step to add a solution to 
some possibility pattern and then extending the ultrafilter to contain the 
solution just added.  We proved by means of a chain condition that the 
Boolean algebra and ultrafilter tailor-made in this way for each one of our given 
theories would fail to saturate any of the other given theories.

\scbr 
\section{The canonical Boolean algebra}

This discussion brings us to a very recent understanding, from \cite{MiSh:F2356}. 

Notice that the second, combinatorial definition of ``possibility pattern'' above does not 
mention ultrafilters.  
In the theorem of the previous section, 
we had built $\ba$ and $\mfd$ together 
by induction so that the algebra and the ultrafilter could together be tailored to $T$.  
We recently understood the algebra can stand alone:\footnote{
Note there have been previous significant Boolean algebras in classification theory, such as in
\cite{Sh:93}. There doesn't seem to be a a strong connection to the canonical Boolean algeba 
at the moment.  
Perhaps this will become clearer in future work. But also, sometimes 
a certain kind of central object may appear earlier in different forms,
as with the relation of splitting, strong splitting, 
weak dividing and so forth to dividing and forking.}  

\begin{quotation} 
\noindent (Informal version) 
To any $T$ and $\lambda$ we can associate the ``canonical Boolean algebra for $T$'' 
which arises essentially by inductively adding formal solutions to all previous 
possibility patterns 
(indexed by $[\lambda]^{<\aleph_0}$, for each formula $\vp$) as generically as possible.  
\end{quotation}
Less informally, given a complete $\ba_\beta$ 
and a possibility pattern $\langle \mb_u : u \in \lao \rangle$ for $T, \vp$,  
the corresponding \emph{single-pattern extension} of $\ba_\beta$ means: 
\begin{enumerate}
\item adding to the generators of $\ba$ a new set $\{ \mb^1_{\{\alpha\}} : \alpha < \lambda \}$
\item adding to the relations of $\ba$ new rules saying that for each finite $u \subseteq \lambda$, 
$\bigcap \{ \mb^1_{ \{ \alpha \} } : \alpha \in u \} \leq \mb_u$. 
\item then taking the completion of the result.  
\end{enumerate}
To make such an extension more obviously canonical, given a complete $\ba_\beta$, a cardinal 
$\lambda$, a theory $T$ and a set of formulas 
$\Delta$, we can define \emph{the} pattern extension $\ba_{\beta+1}$ of $\ba_\beta$ 
by taking the completion of the free product over $\ba_\beta$ (in the natural sense) 
of all single-pattern extensions.  In other words, we formally solve all relevant possibility patterns. 
We can now (informally) define:
\begin{quotation}
\noindent The \emph{canonical Boolean algebra} for $\lambda, T, \Delta$ is the result 
$\ba = \ba_{2^\lambda}$ of the following construction by induction on $\alpha < 2^\lambda$. 
Begin with $\ba_0 = \ba^1_{\aleph_0, \aleph_0, \aleph_0}$, 
the completion of the free Boolean algebra on countably many independent antichains. 
For $\alpha = \beta+1$, let $\ba_{\alpha}$ be the pattern extension of 
$\ba_\beta$, which is by construction complete. 
For $\alpha$ limit, take the union and then take the completion.  
\end{quotation}
It is important to say at the beginning that the word ``canonical'' shouldn't hide or 
constrain the basic flexibility of this construction, which is one of its neatest features. 
For instance, we may prefer that our generation produces positive intersections 
(when not forbidden by the relations) of fewer than $\theta$ elements for some uncountable $\theta$, 
rather than the simplest case of $\theta = \aleph_0$. Or we might want to make a global restriction 
on possibility patterns we solve which has model theoretic meaning, 
such as being based on a given set.    

Here is a first theorem from the book in progress: 

\begin{theorem}[\cite{MiSh:F2356}] 
When $T$ is the random graph, the canonical Boolean algebra $\ba$  
satisfies the following strong chain condition: for $\kappa = \aleph_1$, given any 
sequence of $\kappa$ nonzero elements of $\ba$, there exists a subsequence of size 
$\kappa$ on which any finitely many elements have nonempty intersection.  
When $\theta = \aleph_1$ and $\kappa = {(2^{\aleph_0})^+}$ 
this extends to any simple theory in the case where the formulas 
in $\Delta$ have trivial forking, meaning forking arises only from equality. 
\end{theorem}

\scbr
\section{Disclaimer}

Above we have left out many crucial parts of the story: 
regularity lemmas, the conditional characterization of simplicity, many classes, 
the maximal class (and $\mathfrak{p} = \mathfrak{t}$), the cut spectrum,  
the role of free Boolean algebras, 
many model theoretic developments --  
and of course all of the theorems which, 
at the time of this writing, still belonged to the future.

\vspace{2mm}

\end{document}